\pdfoutput=1
\documentclass[11pt,oneside]{amsart}
\usepackage[
paper=a4paper,
headsep=20pt,text={136mm,208mm},centering,includehead
]{geometry}
\usepackage{graphicx}
\usepackage[dvipsnames,usenames]{xcolor}
\definecolor{cb1}{RGB}{100,143,255}
\definecolor{cb2}{RGB}{220,38,127}
\definecolor{cb3}{RGB}{255,176,0}
\usepackage[colorlinks=true,urlcolor=ForestGreen,linkcolor=ForestGreen,citecolor=ForestGreen]{hyperref}
\hypersetup{
	linkbordercolor={1 0 0}, 
	citebordercolor={0 1 0} 
}
\usepackage[noabbrev]{cleveref}
\usepackage{amsthm}
\usepackage{cite}
\usepackage{wasysym}
\usepackage{amssymb,mathtools,mathrsfs,stmaryrd,mwe}
\usepackage{bbm}
\usepackage[UKenglish]{babel}

\usepackage[backgroundcolor=green,linecolor=green,bordercolor=green]{todonotes}
\usepackage{enumerate}
\usepackage{caption}
\usepackage[labelfont=up,margin=5pt]{subcaption}
\usepackage{wrapfig}
\usepackage{comment}

\usepackage[section]{placeins}
\usepackage{picins}

\usepackage{libertine}
\usepackage[vvarbb,bigdelims,cmintegrals,libertine]{newtxmath}

\iftrue
\makeatletter
\def\@settitle{%
	\vspace*{10pt}
	\begin{flushleft}%
		\LARGE\bfseries
		\strut\@title\strut
	\end{flushleft}%
}
\def\@setauthors{%
	\begingroup
	\def\thanks{\protect\thanks@warning}%
	\trivlist
	\raggedright
	\large \@topsep28\p@\relax
	\advance\@topsep by -\baselineskip
	\item\relax
	\author@andify\authors
	\def\\{\protect\linebreak}%
	\authors
	\ifx\@empty\contribs
	\else
	,\penalty-3 \space \@setcontribs
	\@closetoccontribs
	\fi
	\normalfont
	\endtrivlist
	\endgroup
}
\def\@setaddresses{\par
	\nobreak \begingroup
	\small\raggedright
	\def\author##1{\nobreak\addvspace\smallskipamount}%
	\def\\{\unskip, \ignorespaces}%
	\interlinepenalty\@M
	\def\address##1##2{\begingroup
		Address:
		\@ifnotempty{##1}{(\ignorespaces##1\unskip) }%
		{\ignorespaces##2}\par\endgroup}%
	\def\curraddr##1##2{\begingroup
		\@ifnotempty{##2}{\nobreak\noindent\curraddrname
			\@ifnotempty{##1}{, \ignorespaces##1\unskip}\/:\space
			##2\par}\endgroup}%
	\def\email##1##2{\begingroup
		\@ifnotempty{##2}{\nobreak\noindent E-mail address%
			\@ifnotempty{##1}{\par\ignorespaces##1\unskip}\/:\space
			\ttfamily##2\par}\endgroup}%
	\def\urladdr##1##2{\begingroup
		\def~{\char`\~}%
		\@ifnotempty{##2}{\nobreak\noindent\urladdrname
			\@ifnotempty{##1}{, \ignorespaces##1\unskip}\/:\space
			\ttfamily##2\par}\endgroup}%
	\addresses
	\endgroup
	\global\let\addresses=\@empty
}
\def\@setabstracta{%
	\ifvoid\abstractbox
	\else
	\skip@20pt \advance\skip@-\lastskip
	\advance\skip@-\baselineskip \vskip\skip@
	\box\abstractbox
	\prevdepth\z@ 
	\vskip-22pt
	\fi
}
\renewenvironment{abstract}{%
	\ifx\maketitle\relax
	\ClassWarning{\@classname}{Abstract should precede
		\protect\maketitle\space in AMS document classes; reported}%
	\fi
	\global\setbox\abstractbox=\vtop \bgroup
	\normalfont\small
	\list{}{\labelwidth\z@
		\leftmargin0pc \rightmargin\leftmargin
		\listparindent\normalparindent \itemindent\z@
		\parsep\z@ \@plus\p@
		
	}%
	\item[\hskip\labelsep\bfseries\abstractname.]%
}{%
	\endlist\egroup
	\ifx\@setabstract\relax \@setabstracta \fi
}

\def\ps@headings{\ps@empty
	\def\@evenhead{%
		\setTrue{runhead}%
		\normalfont\scriptsize
		\rlap{\thepage}\hfill
		\def\thanks{\protect\thanks@warning}%
		\leftmark{}{}}%
	\def\@oddhead{%
		\setTrue{runhead}%
		\normalfont\scriptsize
		\def\thanks{\protect\thanks@warning}%
		\rightmark{}{}\hfill \llap{\thepage}}%
	\let\@mkboth\markboth
}\ps@headings

\def\section{\@startsection{section}{1}%
	\z@{-1.4\linespacing\@plus-.5\linespacing}{.8\linespacing}%
	{\normalfont\bfseries\Large}}
\def\subsection{\@startsection{subsection}{2}%
	\z@{-.8\linespacing\@plus-.3\linespacing}{.5\linespacing\@plus.2\linespacing}%
	{\normalfont\bfseries\large}}
\def\subsubsection{\@startsection{subsubsection}{3}%
	\z@{.7\linespacing\@plus.2\linespacing}{-1.5ex}%
	{\normalfont\bfseries}}
\def\@secnumfont{\bfseries}

\renewcommand\contentsnamefont{\bfseries\large}
\def\@starttoc#1#2{\begingroup
	\setTrue{#1}%
	\par\removelastskip\vskip\z@skip
	\@startsection{}\@M\z@{\linespacing\@plus\linespacing}%
	{.5\linespacing}{
		\contentsnamefont}{#2}%
	\ifx\contentsname#2%
	\else \addcontentsline{toc}{section}{#2}\fi
	\makeatletter
	\@input{\jobname.#1}%
	\if@filesw
	\@xp\newwrite\csname tf@#1\endcsname
	\immediate\@xp\openout\csname tf@#1\endcsname \jobname.#1\relax
	\fi
	\global\@nobreakfalse \endgroup
	\addvspace{32\p@\@plus14\p@}%
	\let\tableofcontents\relax
}
\def\contentsname{Contents}
\def\l@section{\@tocline{1}{.5ex}{0mm}{5pc}{}}
\def\l@subsection{\@tocline{2}{0pt}{2em}{5pc}{}}
\makeatother

\fi 

\makeatother

\newtheorem{theorem}{Theorem}[section]
\newtheorem*{corollary*}{}
\newtheorem{proposition}[theorem]{Proposition}
\newtheorem{corollary}[theorem]{Corollary}
\newtheorem{lemma}[theorem]{Lemma}

\newtheorem{question}{Question}
\newtheorem{remark}{Remark}

\Crefname{theorem}{Theorem}{Theorems}
\Crefname{definition}{Definition}{Definitions}
\Crefname{proposition}{Proposition}{Propositions}
\Crefname{corollary}{Corollary}{Corollaries}
\Crefname{question}{Question}{Questions}

\theoremstyle{definition}

\newtheorem{defin}[theorem]{Definition}
\newenvironment{definition}
{\pushQED{\qed}\defin}
{\popQED\enddefin}

\Crefname{defin}{Definition}{Definitions}

\makeatletter

\providecommand{\proofname}{Proof}
\makeatother

\newcommand{\ckh}{{CKh}}

\newcommand{\cdkh}{CDKh}
\newcommand{\dkh}{DKh}
\newcommand{\cdbn}{CDBN}
\newcommand{\dbn}{DBN}

\newcommand{\vup}{v^{u}_+}
\newcommand{\vum}{v^{u}_-}
\newcommand{\vlp}{v^{\ell}_+}
\newcommand{\vlm}{v^{\ell}_-}

\newcommand{\rp}{\mathbb{RP}}
\newcommand{\ru}{r^{u}}
\newcommand{\rl}{r^{\ell}}
\newcommand{\gu}{p^{u}}
\newcommand{\gl}{p^{\ell}}
\newcommand{\au}{a^{u}}
\newcommand{\al}{a^{\ell}}
\newcommand{\bu}{b^{u}}
\newcommand{\bl}{b^{\ell}}
\newcommand{\sgu}{\mathfrak{s}^{u}}
\newcommand{\sgl}{\mathfrak{s}^{\ell}}

\begin{document}
	
	\vspace*{-40pt}
	\title{Some link homologies in \( \mathbb{RP}^3 \)}
	
	\author{William Rushworth}
	\email{\href{mailto:william.rushworth@ncl.ac.uk}{william.rushworth@ncl.ac.uk}}
	\address{School of Mathematics, Statistics and Physics, Newcastle University, United Kingdom}
	
	\def\subjclassname{\textup{2020} Mathematics Subject Classification}
	\expandafter\let\csname subjclassname@1991\endcsname=\subjclassname
	\expandafter\let\csname subjclassname@2000\endcsname=\subjclassname
	
	
	\begin{abstract}
		We introduce extensions of Khovanov homology and the Lee and Bar-Natan spectral sequences for links in \( \mathbb{RP}^3 \). These extensions are distinct to those previously defined by Asaeda-Przytycki-Sikora (and Gabrov\v{s}ek's generalization), Chen, and Manolescu-Willis. The new Lee and Bar-Natan theories each yield Rasmussen invariants (that are distinct to one another). The invariant extracted from the new Lee homology is distinct to that defined by Manolescu-Willis; it is unclear if the same is true for the new Bar-Natan homology and that defined by Chen.
	\end{abstract}
	
	\maketitle
	
	\section{Introduction}\label{Sec:intro}
	This paper introduces extensions of Khovanov, Lee, and Bar-Natan homology to links in \( \rp^3 \). Asaeda-Przytycki-Sikora first defined an such an extension of Khovanov homology over \( \mathbb{Z}_2\), later upgraded to permit \( \mathbb{Z} \)-coefficients by Gabrov\v{s}ek \cite{APS,Gabrovsek2013}. This theory was recently used by Manolescu-Willis to define Lee homology (and an associated Rasmussen invariant) in \( \rp^3 \), among a number of other results \cite{Manolescu25,Ren2024,Ren2025,Yang2025}. Chen defined an alternative extension of Khovanov homology, and used it to define a Bar-Natan theory for nullhomologous links in \( \rp^3 \) \cite{Chen2021,Chen2025}.
	
	In \Cref{Sec:dkh} we define a new extension of Khovanov homology, and demonstrate that it is distinct to the theories due to Asaeda-Przytycki-Sikora, Chen, and Gabrov\v{s}ek. In \Cref{Sec:dlee,Sec:dbn} we define new extensions of the Lee and Bar-Natan spectral sequences, respectively, before showing that they are distinct to those defined by Manolescu-Willis and Chen. Finally, in \Cref{Sec:properties} we extract Rasmussen invariants from the new spectral sequences. These new invariants depend on the ground ring (as in the classical case), and the invariant over \( \mathbb{Q} \) is distinct to that defined by Manolescu-Willis.
	
	Over \( \mathbb{Z}_2 \), although the spectral sequences are distinct in general, the author is not aware of a knot for which Chen's Rasmussen invariant and that defined here carry inequivalent information (excluding homologically nontrivial knots, on which Chen's invariant is not defined). In other words, in all known examples the \( E_{\infty} \) page of either spectral sequence can be determined from that of the other. However, owing to the significant differences in their construction, it is reasonable to suspect that this does not hold in general; see \Cref{Q:pages,Q:z2ras}.
	
	The homology theories defined in this paper share a number of properties with extant theories, such as functoriality, but differ in important ways. For instance, the new theories are not (unoriented) TQFTs \cite{Turaev2006}. It is interesting to consider if there is a generalized TQFT-like structure that accommodates them; see \Cref{Remark:1}.
	
	The extensions of Khovanov and Lee homology defined here are closely related to those given for virtual links in previous work \cite{Rushworth2017}. Although we do not pursue it, it is likely that this relationship can be made precise by considering Gauss codes with additional decoration (see also \cite{Kauffman2024}).
	
	\subsubsection*{Acknowledgements}
	We thank Carlo Collari for many stimulating discussions on the material that would become \Cref{Sec:dbn}.
	
	\section{Preliminaries}\label{Sec:pre}
	
	We begin by collecting some necessary combinatorial background. By \emph{link} we mean an oriented link in \( \rp^3 \) unless stated otherwise (we use the term \emph{classical} to refer to the setting of \(S^3\)).
	
	\begin{definition} A \emph{diagram} of a link is a tangle diagram in \( D^2 \) with antipodal endpoints identified.
	\end{definition}
	An example is given in \Cref{Fig:1}. For a complete discussion on representing links in \( \rp^3 \) by such diagrams see, for example, \cite[Section 2]{Drob90}.
	
	\begin{proposition}[\cite{Drob90}]\label{Prop:rm} Two diagrams represent the same link if and only if they are related by a finite sequence of the classical Reidemeister moves and the following additional moves
		\begin{center}
			\includegraphics[scale=0.75]{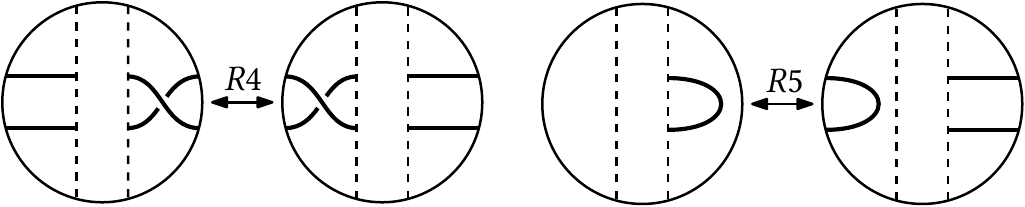}
		\end{center}
	\end{proposition}
	
	Positive and negative crossings, and their \(0\)-\ and \(1\)-resolutions, are defined as in \( S^3 \):
	\begin{center}
		\includegraphics[scale=1]{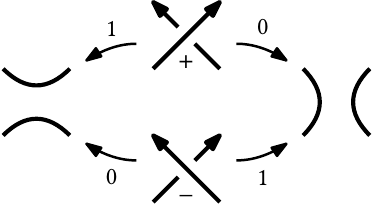}
	\end{center}
	
	We produce a cube of smoothings in the standard way before applying TQFT-like constructions to obtain algebraic invariants.
	
	\begin{definition}
		Let \( D \) be a link diagram. A \emph{smoothing} of \(D\) is a diagram obtained by taking the \(0\)-\ or \(1\)-resolution at every crossing of \(D\).
		
		Suppose that \(D\) has \(n\) crossings. Arbitrarily label the crossings from \(1\) to \(n\), and denote by \( e_1 e_2 \cdots e_n \in \lbrace 0, 1 \rbrace^{\times n} \) the smoothing obtained by resolving the \(i\)-th crossing in the \(e_i\)-resolution. We identify the elements of \( \lbrace 0, 1 \rbrace^{\times n} \) with the corresponding smoothings of \( D \). The resulting object is known as the \emph{cube of smoothings of \(D\)} and is denoted \( \llbracket D \rrbracket \).
	\end{definition}

	An example of a cube of smoothings is given in \cite[Figure 10]{Chen2021}.
	
	If two smoothings of \( D \) differ at exactly one crossing then they cobound a pair of pants or a once-punctured M\"{o}bius band. We regard the corresponding edge of \( \llbracket D \rrbracket \) as being decorated with the appropriate (directed) cobordism. A \emph{face} of \( \llbracket D \rrbracket \) consists of the edges making up a cycle of length \(4\) when \( \lbrace 0 , 1 \rbrace^{\times n} \) is considered as a graph, with elements as vertices and an edge between two vertices if they differ at exactly one position.
	
	We shall frequently make use of the following notions, as they are to our generalizations of Lee and Bar-Natan homology as orientations are to the classical theories. For an expanded discussion in the related setting of virtual links see \cite{Kauffman2004b,Rushworth2021}.
	\begin{definition}[\(2\)-colouring]
		Let \( D \) be a link diagram and \( S(D) \) the immersed curve obtained by removing the over-\ and undercrossing decorations of \(D\).
		
		A \emph{\(2\)-colouring} of \( S(D) \) is a colouring of its nonsingular points using at most \(2\) colours such that the following holds at every double point (up to rotation):
		\begin{center}
			\includegraphics[scale=1]{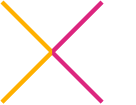}
		\end{center}
		A \emph{\(2\)-colouring} of \( D \) is a \(2\)-colouring of \( S(D) \). A link is \emph{\(2\)-colourable} if a diagram of it possesses a \(2\)-colouring.
	\end{definition}
	Examples of these notions are given in \Cref{Fig:1}. Every knot in \( \rp^3 \) is \(2\)-colourable, as every crossing is visited exactly twice when a diagram is traversed. Not all links are \(2\)-colourable, as demonstrated in \Cref{Fig:1}. The twisted orientable links defined by Chen \cite{Chen2025} form a subset of \(2\)-colourable links; see \Cref{Sec:dbn} for a comparison of the two notions.
	
	\begin{figure}
		\begin{center}
			\includegraphics[scale=0.75]{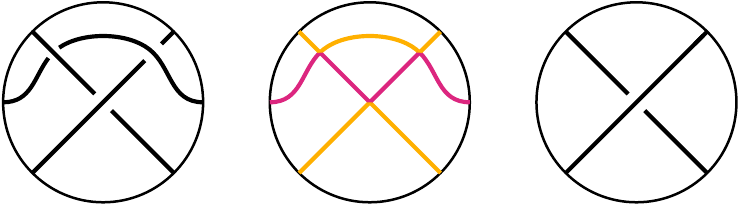}
			\caption{From left to right: a link diagram; a \(2\)-colouring of it; a diagram of a link that is not \(2\)-colourable.\label{Fig:1}}
		\end{center}
	\end{figure}
	
	The \(2\)-colourability of a link is determined as follows.
	
	\begin{definition}\label{Def:degenerate}
		Let \( D = D_1 \cup D_2 \cup \cdots \cup D_n \) be a link diagram with components \( D_i \) (link components rather than connected components in \( \rp^2 \)). A component \( D_i \) is \emph{degenerate} if
		\begin{equation*}
			\text{lk} \left( D_i, D \setminus D_i \right) = 1 \mod 2.
		\end{equation*}
	\end{definition}
	
	Both components of the rightmost diagram of \Cref{Fig:1} are degenerate. It is clear that the moves of \Cref{Prop:rm} preserve the degenerate condition, so that we may speak of degenerate components of a link.
	
	\begin{proposition}\label{Prop:2col}
		Let \( L \) be a link of \(n\) components. Then
		\begin{equation*}
			| \lbrace \text{\(2\)-colourings of \(L\)} \rbrace | = \left\lbrace \begin{matrix}
				0, ~\text{if \(L\) possesses a degenerate component} \\
				2^n, ~\text{otherwise.}
			\end{matrix}\right.
		\end{equation*}
	\end{proposition}
	
	The proof of this statement is easily surmised from those given for the analogous statements in \cite[Section 2]{Rushworth2021}, as Gauss codes are naturally defined for links in \( \rp^3 \).

	Finally, observe that a \(2\)-colouring of a link diagram \(D\) prescribes a smoothing of \(D\) by resolving crossings via the convention:
	\begin{equation}\label{Eq:2colsmoothing}
			\begin{matrix}
			\includegraphics[scale=1]{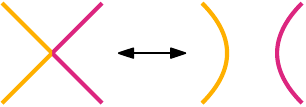}
			\end{matrix}
	\end{equation}
	We refer to smoothings obtained in this way as \emph{\(2\)-coloured smoothings}.
	
	\section{Extending Khovanov homology}\label{Sec:dkh}
	The homology theories defined in \Cref{Sec:dkh,Sec:dlee,Sec:dbn} are constructed via the now-standard procedure: given a diagram \(D\) we associate modules to the vertices and maps to the edges of \( \llbracket D \rrbracket \) to produce a chain complex, the chain homotopy equivalence class of which is an invariant of the link represented by \(D\). In this section we do so to construct an extension of Khovanov homology to \( \rp^3 \).
	
	Let \( R \) be a commutative unital ring and set
	\begin{equation*}
		\mathcal{A} = R \left[ X \right] / \langle X^2 \rangle  = \langle v_+, v_- \rangle_R.
	\end{equation*}
	Define a grading, \(p\), on \( \mathcal{A} \) as \( p ( v_{\pm} ) = \pm 1 \) and extend it linearly over tensor products. Let \( \lbrace - 1 \rbrace \) denote a shift in \( p \)-grading by \( -1 \) and set
	\begin{equation*}
		A^k = \mathcal{A}^{ \otimes k } \oplus \mathcal{A}^{ \otimes k } \lbrace - 1 \rbrace
	\end{equation*}
	so that in particular \( A \coloneqq A^1 = \mathcal{A} \oplus \mathcal{A} \lbrace - 1 \rbrace \). Where necessary we distinguish between the unshifted and shifted copies of \( \mathcal{A} \) via the superscripts \( u \) and \( \ell \) e.g.\
	\begin{equation*}
		A = \langle \vup, \vum, \vlp, \vlm \rangle_R .
	\end{equation*}
	We abbreviate super-\ and subscripts across tensor products, writing \( v^u_{+-} = \vup \otimes \vum \) and so on.
	
	\begin{definition}\label{Def:cdkh}
		Let \(D\) be a link diagram with \(n_- \) negative crossings. To \(S\), a vertex of \( \llbracket D \rrbracket \) consisting of \(k\) circles, assign \( A^k \), under an arbitrary identification of the circles of \(S\) with (pairs of) the tensorands of \(A^k\). Define the \emph{height of \(S\)} as
		\begin{equation*}
			\left| S \right| \coloneqq \# \left( \text{\(1\)-resolutions in \(S\)} \right) - n_-
		\end{equation*}
		and set \( \cdkh_i (D) \) to be the direct sum of the \(R\)-modules assigned to smoothings of height \(i\). We refer to \(i\) as the homological grading.
		
		Postponing verification to \Cref{Prop:d2}, we assert that the \emph{doubled Khovanov complex of \(D\)} is the chain complex with chain spaces \( \cdkh_i (D) \) and differential given by matrices of maps, the entries of which are determined by the edges of \( \llbracket D \rrbracket \). The maps \( m \), \( \Delta \) are assigned to pairs of pants, and \( \eta \) to the once-punctured M\"{o}bius band. Their specific form is as follows.
		
		The \( m \) and \( \Delta \) maps do not interact with the \( u \), \( \ell \) superscripts and so we suppress them
		\begin{equation}\label{Eq:diffcomp}
			\begin{aligned}
			m : A^2 &\rightarrow A &\qquad \Delta : A &\rightarrow A^2 \\
			v_{++} &\mapsto v_+ &\qquad v_+ &\mapsto v_{+-} + v_{-+} \\
			v_{+-}, v_{-+} &\mapsto v_{-} &\qquad v_{-} &\mapsto v_{--}. \\
			v_{--} &\mapsto 0 &
			\end{aligned}
		\end{equation}
		(These are the maps employed in classical Khovanov homology.)
		
		The \(\eta\) map has the form
		\begin{equation*}
			\begin{aligned}
				\eta : A &\rightarrow A \\
				\vup &\mapsto \vlp \\
				\vlp &\mapsto 2 \vum \\
				\vum &\mapsto \vlm \\
				\vlm &\mapsto 0.
			\end{aligned}
		\end{equation*}
		Its extension across tensor products alters the superscript of the entire string and the subscript of the tensorand associated to the boundary of the once-punctured M\"{o}bius band. E.g.\ \( \eta ( v^{u}_{++-} ) = 2 v^{\ell}_{+--} \) if the boundary in question corresponds to the second tensorand in both domain and codomain.
		
		To these entries we add signs, using any convention that ensures that each face contains an odd number of minus signs.
		
		Finally, we define the quantum grading on \( \cdkh (D) \) as
		\begin{equation*}
			j ( x ) = p ( x ) + i (x) + wr (D)
		\end{equation*}
		for \(i\) the homological grading and \( wr (D) \) the writhe of \(D\), so that
		\begin{equation*}
			\cdkh ( D ) = \bigoplus_{i,j} \cdkh_{i,j} ( D )
		\end{equation*}
		for \( \cdkh_{i,j} ( D ) \) the homogeneous elements of homological degree \(i\) and quantum degree \(j\).
	\end{definition}
	
	\begin{proposition}\label{Prop:d2}
		When equipped with the differential as described in \Cref{Def:cdkh} the object \( \cdkh ( D )  \) is a chain complex.
	\end{proposition}
	
	\begin{proof}
		Unlike other extensions of Khovanov homology to \( \rp^3 \) (such as \cite{APS,Gabrovsek2013,Chen2021}) the doubled construction does not involve any auxiliary data in addition to those used in the construction of classical Khovanov homology. That is, no extra decorations are placed on diagrams or their smoothings.
		
		As such we need only consider the maps around the faces of \( \llbracket D \rrbracket \), rather than the particular diagrams that yield them. It suffices to check the faces
		\begin{equation*}
			\begin{matrix}
				\includegraphics[scale=0.85]{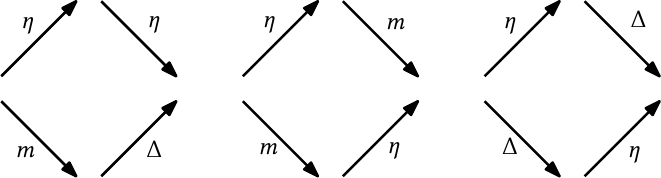}
			\end{matrix}
		\end{equation*}
		as the other possible faces are handled by the classical proof. The commutativity of these faces is readily verified, so that any sign assignment yields a chain complex.
	\end{proof}
	
	\begin{theorem}\label{Thm:invariance1}
		The chain homotopy equivalence class of \( \cdkh (D) \) is an invariant of the link represented by \( D \).
	\end{theorem}
	
	\begin{proof}
		As with the proof of \Cref{Prop:d2}, that we have used no additional auxiliary data when constructing \( \cdkh (D) \) allows us to follow the classical proof of invariance closely.
		
		The moves \( R4 \) and \( R5 \) of \Cref{Prop:rm} induce isomorphisms at the chain level. That is, neither move alters the number of circles in a smoothing nor the configuration of edges between them.
		
		Concerning the classical Reidemeister moves, we suffice ourselves by observing that the so-called delooping techniques used to prove invariance in the classical case translate in a straight-forward way to this setting (see, for example, \cite[Lemma 3.7]{Bar-Natan2002}, \cite[Lemma 3.2]{Bar-Natan2005}, and \cite[Proposition 5.1]{Naot2006}).
	\end{proof}
	
	It follows that the homology of \( \cdkh ( D ) \) is an invariant of the link, \(L\), represented by \( D \), and we define the \emph{doubled Khovanov homology of \(L\)} to be
	\begin{equation*}
		\dkh ( L ) \coloneqq H \left( \cdkh (D) \right).
	\end{equation*}
	
	\subsection*{Comparison with other extensions}
	To compare doubled Khovanov homology to previously-defined extensions it is convenient to pass to a reduced version. Let \( D \) be a diagram of a link \( L \) with a marked point on one of its components. The marked point may be carried over to smoothings of \( D \); a circle in a smoothing is itself \emph{marked} if it contains the marked point. Following the classical definition, define \( \widetilde{\cdkh  }(D) \) to be the subcomplex of \( \cdkh (D) \) spanned by states in which the tensorand corresponding to the marked circle is labelled with \( \vum \) or \( \vlm \). Denote by \( \widetilde{\dkh  }(L) \) the homology of \( \widetilde{\cdkh  }(D) \). The proof that \( \widetilde{\dkh  }(L) \) is indeed an invariant of \( L \) is essentially identical to that applied in the classical case.
	
	Consider the following knot, labelled \( 3_1 \) in Drobotukhina's table \cite{Drobotukhina94}, and its (reduced) doubled Khovanov homology over \( \mathbb{Z}_2 \), where \( \bullet, \circ \) denotes a copy of \( \mathbb{Z}_2 \), the hollow generator \( \circ \) is at bigrading \( (0,0) \), and both axes are labelled by the integers: 
	\begin{center}
		\includegraphics[width=1\textwidth]{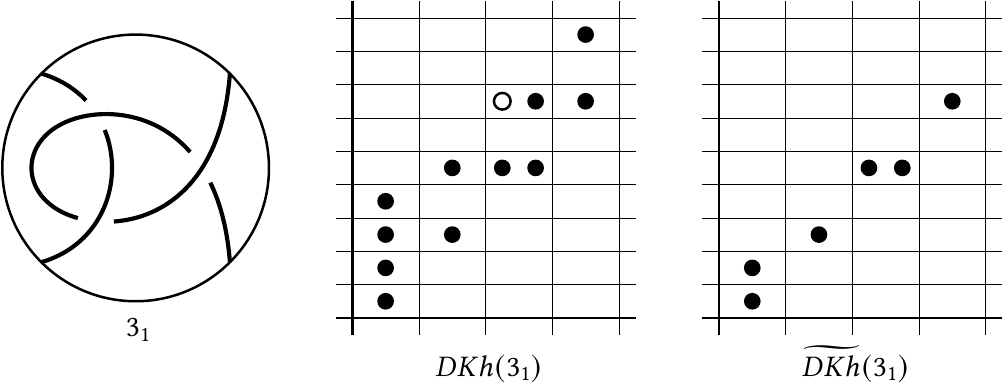}
	\end{center}
	In \cite[Section 2.6]{Chen2021} a range of homology theories (with coefficients in \( \mathbb{Z}_2 \)) are evaluated on \( 3_1 \). The reduced doubled Khovanov homology of \( K \) is distinguished from \( \alpha_0 \), \( \alpha_1 \), \( \alpha_{HF} \), \( \alpha^{\prime}_{HF} \), and from reduced Asaeda-Przytycki-Sikora homology by the rank.

	\begin{wrapfigure}{r}{0.32\textwidth}
		\centering
		\includegraphics[scale=0.65]{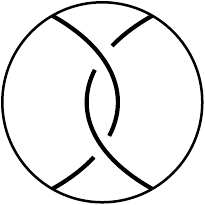}
		\caption{}\label{Fig:2}
	\end{wrapfigure}
	
	Gabrov\v{s}ek defined an extension of Asaeda-Przytycki-Sikora homology to allow for \( \mathbb{Z} \)-coefficients \cite{Gabrovsek2013}, to which doubled Khovanov homology is also distinct. The knot of \Cref{Fig:2} is labelled \(2_1\) in Drobotukhina's table; evaluating Gabrov\v{s}ek's theory on \( 2_1 \) yields no torsion, while its doubled Khovanov homology over \( \mathbb{Z} \) contains \(2\)-torsion.
	
	\section{Extending Lee homology}\label{Sec:dlee}
	
	Lee defined a perturbation of Khovanov homology in which terms of nonzero quantum degree are added to the differential \cite{Lee2005}. The rank and homological degree support of the resulting theory depend only on the linking matrix, but the quantum degree support contains deep geometric information. The perturbed theory is also a concordance invariant.
	
	In this section we define an extension of Lee homology to \( \rp^3 \) that also possesses the properties described above (with concordance invariance established in \Cref{Sec:properties}).
	
	\begin{definition}\label{Def:cdlee}
		Let \( D \) be a link diagram. Setting \( R \) to be commutative unital ring in which \(2\) is invertible, denote by \( \cdkh' (D) \) the chain complex with chain spaces equal to those of \( \cdkh (D) \) (as given in \Cref{Def:cdkh}) and differential components
		\begin{align*}
			m' : A^2 &\rightarrow A & \Delta ' : A &\rightarrow A^2 \\
			v_{++} &\mapsto v_+ & v_+ &\mapsto v_{+-} + v_{-+} \\
			v_{+-}, v_{-+} &\mapsto v_{-} & v_{-} &\mapsto v_{--} + v_{++} \\
			v_{--} &\mapsto v_+ &
		\end{align*}
		(where we have suppressed superscripts as in \Cref{Eq:diffcomp}) and
		\begin{equation*}
			\begin{aligned}
				\eta' : A &\rightarrow A \\
				\vup &\mapsto \vlp \\
				\vlp &\mapsto 2 \vum \\
				\vum &\mapsto \vlm \\
				\vlm &\mapsto 2 \vup.
			\end{aligned}
		\end{equation*}
		Notice that these maps possess terms of \(j\)-degree \(4\), so that the \(j\)-grading of \( \cdkh (D) \) must be weakened to a filtration of \( \cdkh ' (D) \). The homological grading is defined as is done for \( \cdkh (D) \).
	\end{definition}
	
	\begin{theorem}\label{Thm:dleeinvariance}
		Let \( D \) be a diagram of a link \(L\). The object \( \cdkh ' ( D ) \) is a chain complex, the chain homotopy equivalence class of which is an invariant of \(L\).
	\end{theorem}
	\begin{proof}
		As with similar verifications in \Cref{Sec:dkh} the classical proof can be followed closely, owing to the lack of auxiliary data employed in the construction of \( \cdkh ' ( D ) \). See, for example, \cite[Section 6]{Rasmussen2010}.
	\end{proof}
	
	Given \(D\) a diagram of \(L\) we define the \emph{doubled Lee homology of \(L\)} to be
	\begin{equation*}
		\dkh' ( L ) \coloneqq H \left( \cdkh' (D) \right).
	\end{equation*}
	
	Manolescu-Willis defined an extension of Lee homology to \( \rp^3 \) \cite{Manolescu25}, to which doubled Lee homology is distinct. For example, the result of evaluating Manolescu-Willis' extension on the knot \(2_1\) (given in \Cref{Fig:2}) is supported in homological degree \(0\); the doubled Lee homology of \(2_1\) is supported in homological degree \(-2\). The two homologies are also distinguished by their quantum degree supports, as discussed in \Cref{Sec:ras_inv}.
	
	As in the classical case doubled Lee homology is the \( E_{\infty} \) page of the spectral sequence associated to the filtration implicit in \Cref{Def:cdlee}.
	
	\begin{theorem}\label{Thm:leess}
		Let \( L \) be a link. There is a spectral sequence with \( E_2 \) page \( \dkh (L) \) that abuts to \( \dkh ' ( L ) \), each page of which is an invariant of \( L \).
	\end{theorem}

	\begin{proof}
		The argument is essentially identical to that given in the classical case, owing to the fact that 
		\( \eta ' \) is the sum of maps of \(j\)-degree \(0\) and \(4\). See again \cite[Section 6]{Rasmussen2010}.
	\end{proof}

	\subsection*{Canonical generators}
	
	Classical Lee homology has a canonical basis consisting of (generators associated to) the orientations of the argument link. This follows from the fact that an oriented smoothing of a link diagram in the plane has the following property. Let \( D \) be a diagram and \(S\) a smoothing of it; we say that \(S\) is \emph{bipartite} if its Tait graph is bipartite. An \(n\)-component classical link has \(2^{n-1}\) bipartite smoothings, that are easily put into bijection with its orientations up to global reversal. 
	
	Doubled Lee homology also has a canonical basis constructed from bipartite smoothings, but the generators are no longer in bijection with the orientations of the argument link. In fact, there are links in \( \rp^3 \) with no bipartite smoothings. As described in \Cref{Eq:2colsmoothing}, a \(2\)-colouring of a link diagram \(D\) prescribes a smoothing of \(D\). It is clear that a \(2\)-coloured smoothing is bipartite, and that a bipartite smoothing can be \(2\)-coloured in exactly two ways (related by switching the colours on all circles)\footnote{Of course, any distinction between \(2\)-coloured and bipartite is almost tautological: a \(2\)-colouring of a smoothing is simply a partition of its circles that certifies the Tait graph as bipartite. We choose to write in terms of \(2\)-colourings as this simplifies exposition.}.
	
	\Cref{Prop:2col} therefore implies that a link of \(n\) components in \( \rp^3 \) has \(0\) or \(2^{n-1}\) bipartite smoothings. We can recover important results regarding the rank and homological degree support of doubled Lee homology with \(2\)-colourings taking the place of orientations. Henceforth we assume that \(2\)-colourings employ the colours {\color{cb3}orange} and {\color{cb2}pink}.
	
	First we define the basis of \( A \) given by \( \lbrace \ru, \rl, \gu, \gl \rbrace \) where
	\begin{equation*}
		\begin{aligned}
			\ru &= \dfrac{\vup + \vum}{2} \\
			\gu &= \dfrac{\vup - \vum}{2}
		\end{aligned}
	\end{equation*}
	and \( \rl \), \( \gl \) are defined similarly (this is a suitably adapted version of the \( \lbrace a, b \rbrace \) basis given in \cite[Section 4.3.1]{Lee2005}). To \(C\) a \(2\)-colouring of a diagram \( D \) we associate the element \( \sgu_C \in \cdkh ' ( D ) \) defined as
	\begin{equation*}
		\sgu_C = x_1 \otimes x_2 \otimes \cdots \otimes x_k
	\end{equation*}
	where \( x_i = \ru \) if the \(i\)-th circle of \( C \) is coloured orange, and \( x_i = \gu \) otherwise. The element \( \sgl_C \) is defined similarly. We often suppress the dependence on the \(2\)-colouring \( C \), writing simply \( \sgu \), \( \sgl \).
	
	The elements \( \sgu \), \( \sgl \) are known as \emph{canonical generators}, terminology justified by the following result.
	
	\begin{theorem}\label{Thm:cangen}
		Let \( D \) be a diagram of \( L \). The set of canonical generators associated to the \(2\)-colourings of \(D\) generate \( \dkh ' ( L ) \). Moreover, the map on \( \dkh ' (L) \) induced by a Reidemeister move sends canonical generators to canonical generators (up to sign).
	\end{theorem}
	
	\begin{proof}
		Again, the lack of auxiliary data employed in the construction of \( \dkh ' ( L ) \) allows us to follow the classical proofs. Additionally, the proofs of analogous statements in the closely related setting of virtual links apply here \cite[Section 3]{Rushworth2017}. We suffice ourselves by observing that the differential components have the following form in the \( \lbrace \ru, \rl, \gu, \gl \rbrace \) basis:
		\begin{equation*}
			\begin{aligned}
				m' (r \otimes r) &= r &\qquad \Delta ' (r) &= 2 r \otimes r \\
				m' (r \otimes p) = m' (p \otimes r) &= 0 &\qquad \Delta ' (p) &= 2 p \otimes p \\
				m' (p \otimes p) &= p & &
			\end{aligned}
		\end{equation*}
		and
		\begin{equation*}
			\begin{aligned}
				\eta ' ( \ru ) &= \rl \\
				\eta ' ( \rl ) &= 2\ru \\
				\eta ' ( \gu ) &= \gl \\
				\eta ' ( \gl ) &= 2\gu.
			\end{aligned}
		\end{equation*}
		The proof proceeds by combining the above with the fact that \(2\)-colourable smoothings have only \( \Delta \) maps incoming and \( m \) maps outgoing. 
	\end{proof}

	Combining this with \Cref{Prop:2col} we obtain the following.
	
	\begin{corollary}\label{Cor:rank}
		Let \( L \) be a link of \(n\) components. Then
		\begin{equation*}
			\text{rank} \left( \dkh ' ( L ) \right) = \left\lbrace \begin{matrix}
				0, ~\text{if \(L\) has a degenerate component} \\
				2^{n+1}, ~\text{otherwise.}
			\end{matrix}\right.
		\end{equation*}
	\end{corollary}
	
	Notice that the rank is \( 2^{n+1} \), rather than \( 2^n\) as in the classical case, as to each \(2\)-colouring \( C \) we associate two canonical generators, \( \sgu_C \) and \(\sgl_C\).

	Given a \(2\)-colouring \(C\) of a diagram \(D\) we declare a crossing of \(D\) to be \emph{odd} if it is in the configuration
	\begin{equation}\label{Eq:odd_crossing}
		\begin{matrix}
		\includegraphics[scale=1]{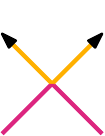}
		\end{matrix}
	\end{equation}
	up to globally switching the colours. Otherwise the crossing is declared to be \emph{even}. Notice that for crossings between distinct link components this designation depends on the \(2\)-colouring. 
	
	\begin{corollary}\label{Cor:homsupport}
		Let \(D\) be a diagram of a link \(L\). Given \( C \) a \(2\)-colouring of \(D\), denote by \( n^o_+ \) and \(n^o_-\) the number of odd positive and odd negative crossings, respectively.  The homological degree support of \( \dkh ' (L) \) is given by
		\begin{equation*}
			\left\lbrace n^o_+ - n^o_- ~|~ \text{\(C\) a \(2\)-colouring of \(D\)} \right\rbrace.
		\end{equation*}
	\end{corollary}
	
	Identifying crossings as even or odd in this way is an instance \emph{parity}; see \cite{Ilyutko2013,Manturov2010} and references therein. The quantity \( n^o_+ - n^o_- \) is an example of an \emph{odd writhe} in this context.

	\begin{proof}
		Consulting \Cref{Eq:odd_crossing} we see that a crossing is resolved into its \(1\)-resolution in the the \(2\)-coloured smoothing associated to \(C\) if and only if it is an even negative or odd positive crossing. Setting \(n_-\) to the total number of negative crossings and \( n^e_- \) the number of even negative crossings of \( D \), it follows that
		\begin{equation*}
			i \left( \sgu_C \right) = i \left( \sgl_C \right) = n^o_+ + n^e_- - n_- = n^o_+ - n^o_-.
		\end{equation*}
	\end{proof}
	
	\section{Extending Bar-Natan homology}\label{Sec:dbn}
	As with its classical counterpart, doubled Lee homology is restricted to ground rings in which \(2\) is invertible. Bar-Natan introduced a companion theory to Lee homology with coefficients in \( \mathbb{Z}_2 \) \cite[Section 9.3]{Bar-Natan2005}. Here we extend Bar-Natan's theory to \(\rp^3 \), and compare it to the extension previously defined by Chen \cite{Chen2025}.
	
	\begin{definition}\label{Def:cdbn}
		Let \( D \) be a link diagram. Setting \( R = \mathbb{Z}_2 \), denote by \( \cdbn (D) \) the chain complex with chain spaces equal to those of \( \cdkh (D) \) and differential components:
		\begin{align*}
		m_2 : A^2 &\rightarrow A & \Delta_2 : A &\rightarrow A^2 \\
		v_{++} &\mapsto v_+ & v_+ &\mapsto v_{+-} + v_{-+} \\
		v_{+-}, v_{-+} &\mapsto v_{-} & v_{-} &\mapsto v_{--} \\
		v_{--} &\mapsto v_- &
		\end{align*}
		(where we have again suppressed superscripts) and
		\begin{equation*}
		\begin{aligned}
		\eta_2 : A &\rightarrow A \\
		\vup &\mapsto \vlp \\
		\vlp &\mapsto \vup \\
		\vum &\mapsto \vlm \\
		\vlm &\mapsto \vum.
		\end{aligned}
		\end{equation*}
		Notice that these maps possess terms of \(j\)-degree \(2\), so that the \(j\)-grading of \( \cdkh (D) \) must be weakened to a filtration of \( \cdbn (D) \). The homological grading is defined as is done for \( \cdkh (D) \).
	\end{definition}

	\begin{theorem}\label{Thm:dbninvariance}
		Let \( D \) be a diagram of a link \(L\). The object \( \cdbn ( D ) \) is a chain complex, the chain homotopy equivalence class of which is an invariant of \(L\).
	\end{theorem}

	\begin{proof}
		This verification is essentially identical to that of the analogous result for \( \cdkh ' (D) \).
	\end{proof}

	Given \(D\) a diagram of \(L\) we define the \emph{doubled Bar-Natan homology of \(L\)} to be
	\begin{equation*}
	\dbn ( L ) \coloneqq H \left( \cdbn (D) \right).
	\end{equation*}
	
	There is a natural analogue of \Cref{Thm:leess} in this setting.
	
	\begin{theorem}\label{Thm:dbnss}
		Let \( L \) be a link. There is a spectral sequence with \( E_2 \) page \( \dkh (L) \) that abuts to \( \dbn ( L ) \), each page of which is an invariant of \(L\).
	\end{theorem}
	
	We refer to this sequence as the \emph{doubled Bar-Natan spectral sequence}.
	
	\subsection*{Canonical generators}
	
	Doubled Bar-Natan is structurally very similar to doubled Lee homology. Here we establish the analogue of \Cref{Thm:cangen} for doubled Bar-Natan homology (leaving concordance invariance to \Cref{Sec:properties}).
	
	The set \( \lbrace \ru, \rl, \gu, \gl \rbrace \) of \Cref{Sec:dlee} is not a basis if coefficients are taken in \( \mathbb{Z}_2 \). We therefore use the following basis \( \lbrace \au, \al, \bu, \bl \rbrace \) where
	\begin{equation*}
		\begin{aligned}
			\au &= \vup + \vum \\
			\bu &= \vum
		\end{aligned}
	\end{equation*}
	and \( \al \), \( \bl \) are defined similarly (see \cite[Section 3]{Turner2006}).
	
	We define the \emph{canonical generators} of \( \dbn (L) \) in almost identical fashion to those of \( \cdkh ' ( L ) \). That is, to \(C\) a \(2\)-colouring of a diagram \( D \) we associate the element \( \sgu_C \in \cdbn ( D ) \), given by
	\begin{equation*}
	\sgu_C = x_1 \otimes x_2 \otimes \cdots \otimes x_k
	\end{equation*}
	where \( x_i = \au \) if the \(i\)-th circle of \( C \) is coloured orange, and \( x_i = \bu \) otherwise. The element \( \sgl_C \) is defined similarly. As in \Cref{Sec:dlee} we often suppress the dependence on the \(2\)-colouring \( C \), writing simply \( \sgu \), \( \sgl \).
	
	\begin{theorem}\label{Thm:cangen2}
		Let \( D \) be a diagram of \( L \). The set of canonical generators associated to the \(2\)-colourings of \(D\) form a basis for \( \dbn ( L ) \). Moreover, the map on \( \dbn (L) \) induced by a Reidemeister move sends canonical generators to canonical generators.
	\end{theorem}

	\begin{proof}
		We can proceed exactly as in the classical case, in light of the following observation:
		\begin{equation*}
			\begin{aligned}
				\eta_2 ( \au ) &= \al \\
				\eta_2 ( \al ) &= \au \\
				\eta_2 ( \bu ) &= \bl \\
				\eta_2 ( \bl ) &= \bu.
			\end{aligned}
		\end{equation*}
	\end{proof}

	It follows that the statements obtained from \Cref{Cor:rank,Cor:homsupport} by replacing \( \dkh ' (L) \) with \( \dbn (L) \) hold also.
	
	\begin{corollary}\label{Cor:rankbn}
		Let \( L \) be a link of \(n\) components. Then
		\begin{equation*}
		\text{rank} \left( \dbn ( L ) \right) = \left\lbrace \begin{matrix}
		0, ~\text{if \(L\) has a degenerate component} \\
		2^{n+1}, ~\text{otherwise.}
		\end{matrix}\right.
		\end{equation*}
	\end{corollary}

	\begin{corollary}\label{Cor:homsupportbn}
		Let \(D\) be a diagram of a link \(L\). Given \( C \) a \(2\)-colouring of \(D\), denote by \( n^o_+ \) and \(n^o_-\) the number of odd positive and odd negative crossings, respectively.  The homological degree support of \( \dbn (L) \) is given by
		\begin{equation*}
		\left\lbrace n^o_+ - n^o_- ~|~ \text{\(C\) a \(2\)-colouring of \(D\)} \right\rbrace.
		\end{equation*}
	\end{corollary}
	
	\subsection*{Comparison with Chen's extension}
	
	Chen defined an extension of the Bar-Natan spectral sequence for twisted orientable links in \( \rp^3 \) \cite{Chen2025}. Here we provide examples demonstrating that Chen's spectral sequence and that of \Cref{Thm:dbnss} are distinct. On the other hand we show that, forgetting the quantum grading, the \( E_{\infty} \) pages of these sequences carry equivalent information. It is unknown to the author if the \( E_{\infty} \) pages are distinguished by their quantum grading in general; see \Cref{Sec:ras_inv}.
	
	\begin{definition}[{\cite[Definition 2.8]{Chen2025}}]\label{Def:twiori}
		A link is \emph{twisted orientable} if it is made up of nullhomologous components. A \emph{twisted orientation} of a link diagram \(D\) is an orientation of the arcs of \( D \) such that the orientations disagree at antipodal points (on the boundary of the ambient disc).
	\end{definition}
	
	An example is given in \Cref{Fig:3}. Twisted orientations are to Chen's extension as orientations are to classical Bar-Natan homology. Recall that the same is true for \(2\)-colourings and doubled Bar-Natan homology. We therefore begin by comparing the notions of twisted orientation and \(2\)-colouring.
	
	Let \( D \) be a twisted oriented diagram. The twisted orientation prescribes a smoothing in identical fashion to an orientation, so that we may speak of a \emph{twisted oriented resolution} of a crossing of \(D\); taking the twisted oriented resolution at every crossing yields a \emph{twisted oriented smoothing of \(D\)}. It is implicit in the proof of \cite[Proposition 2.16]{Chen2025} that twisted oriented smoothings are bipartite. As described in \Cref{Eq:2colsmoothing}, bipartite smoothings correspond to \(2\)-colourings of \(D\), up to globally switching the colours. It follows that a twisted orientable link is \(2\)-colourable.
	
	The converse is false in general: the link on the left of \Cref{Fig:1} is \(2\)-colourable but not twisted orientable. It follows that this doubled Bar-Natan homology is defined for such a link while Chen's extension is not.
	
	However, the two notions are equivalent when restricted to nullhomologous links.
	\begin{lemma}\label{Lem:to2cl}
		Let \(L\) be a nullhomologous link. If \(L\) is \(2\)-colourable then it is twisted orientable.
	\end{lemma}
	
	\begin{proof}
		Suppose that \( L = L_1 \cup \cdots \cup L_n \) and let \( D = D_1 \cup \cdots \cup D_n \) be a diagram of \(L\). If \(L\) is not twisted orientable then a nonzero (necessarily even) number of its components are homologically nontrivial; denote them \( L_{e_1}, \ldots, L_{e_{2k}} \).
		
		As \( D_{e_i} \) and \( D_{e_j} \) represent homologically nontrivial components there is a sequence of crossing changes that converts \( D \) to \( D' = D^{\prime}_1 \cup \cdots \cup D^{\prime}_n \), where \( D_i = D^{\prime}_i \) away from crossing neighbourhoods and
		\begin{equation*}
			\text{lk} \left( D^{\prime}_{e_i}, D^{\prime}_{e_j} \right) = 1
		\end{equation*}
		for all \( 1 \leq i < j \leq 2k \).
		
		As crossing changes preserve linking number modulo \(2\) we have
		\begin{equation*}
			\begin{aligned}
				\text{lk} \left( D_{e_i}, D \setminus D_{e_i} \right) \mod 2 &= \sum_{i = 1,~ i \neq j}^{2k} \text{lk} \left( D_{e_i}, D_{e_j} \right) \mod 2 \\
				&= 1 \mod 2.
			\end{aligned}
		\end{equation*}
		Thus \(D\) has a degenerate component and \(L\) is not \(2\)-colourable.
	\end{proof}
	
	\subsubsection*{Rank and homological degree support}
	
	Let \(L\) be a nullhomologous link and \( HBN(L) \) denote Chen's extension of Bar-Natan homology. Here we establish that the rank and homological degree support of \( HBN(L) \) and those of \( \dbn ( L ) \) carry equivalent information.
	
	\begin{proposition}\label{Prop:chdbn_rank}
		Let \( L \) be a nullhomologous link. Then the rank of \( \dbn (L) \) is twice that of \( HBN ( L ) \).
	\end{proposition}
	\begin{proof}
		Combine \Cref{Cor:rankbn}, \cite[Theorem 1.1]{Chen2025}, and \Cref{Lem:to2cl}.
	\end{proof}

	It follows that if \( L \) is nullhomologous and not twisted orientable then the homological degree supports of both \( HBN(L) \) and \( \dbn ( L ) \) are empty. It remains to relate the supports of the theories in the twisted orientable case. We introduce the following notions to do so, owing to the different grading conventions used here and in \cite{Chen2025}.
	
	Suppose that \( L \) is twisted orientable. Placing a marked point on each component of \(L\) induces a bijection between the twisted orientations of \(L\) and its orientations: a twisted orientation is sent to the unique orientation with which it agrees at the marked points. Henceforth we shall use one such bijection implicitly; different choices of marked points are accounted for in the shifts given in \Cref{Prop:chdbn_support} below.
	
	\begin{figure}
		\begin{center}
			\includegraphics[scale=0.8]{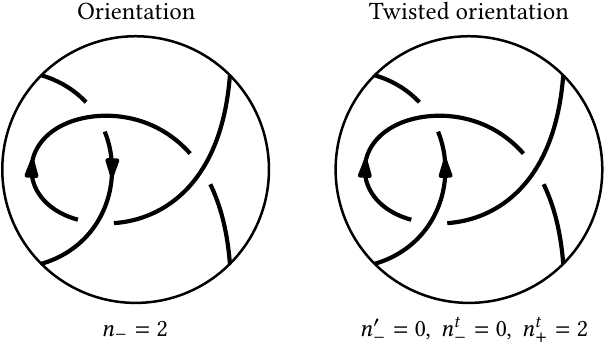}
			\caption{An oriented diagram associated to a twisted oriented diagram. In the twisted oriented diagram the two leftmost crossings are themselves twisted.\label{Fig:3}}
		\end{center}
	\end{figure}
	
	Let \( D \) be a twisted oriented diagram. We say that a crossing is itself \emph{twisted} if its twisted oriented resolution differs from its oriented resolution. Denote by \( n^t_{+} \) and \( n^t_{-} \) the number of twisted positive and twisted negative crossings, respectively, with positive and negative dictated by the twisted orientation. See again \Cref{Fig:3}.
	
	Given a \(2\)-colouring of \(D\) we designate crossings as either even or odd as per \Cref{Eq:odd_crossing}. There is a natural correspondence between twisted orientations and \(2\)-colourings of \(D\) (up to globally switching colours), under which the notions of twisted and odd crossings are equivalent.
	
	\begin{lemma}\label{Lem:twistedodd}
		Let \(D\) be a twisted oriented diagram. The twisted orientation prescribes a \(2\)-colouring of \(D\), up to globally switching the colours. A crossing is twisted with respect to the twisted orientation if and only if it is odd with respect to the associated \(2\)-colouring.
	\end{lemma}

	\begin{proof}
	The twisted orientation defines a twisted oriented smoothing of \(D\). This smoothing is bipartite so that it can be \(2\)-coloured in two ways, that are related by globally switching the colours. Both of these smoothings define \(2\)-colourings of \(D\) via \Cref{Eq:2colsmoothing}.
	
	Without loss of generality let \(C\) denote one of these \(2\)-colourings (the argument is not affected by globally switching colours). To see that a crossing is twisted if and only if it is odd with respect to \(C\) we consult the following sequence of diagrams
	\begin{center}
		\includegraphics[scale=1.2]{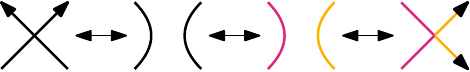}
	\end{center}
	Here the leftmost crossing has its twisted orientation, and the rightmost its orientation (recall that we are implicitly using a bijection between twisted orientations and orientations of \(D\)). The \(2\)-colouring must look as in the second-right diagram (up to globally switching colours) as twisted oriented smoothings are bipartite.
	\end{proof}

	Let \( C \) be a \(2\)-colouring of a diagram \(D\) (with positive and negative crossings defined with respect to the orientation, not a twisted orientation). Let \( n^e_{+}, n^e_{-} \) denote the number of even positive and negative crossings, respectively, and \( n^o_{+}, n^o_{-} \) the number of odd positive and negative crossings, respectively, where the parity is taken with respect to \(C\).

	\begin{proposition}\label{Prop:chdbn_support}
		Let \( D \) be a twisted oriented diagram of \( L \). Then \( HBN_i ( L ) \neq 0 \) if and only if \( \dbn_{i+n^t_{-}-n^t_{+}} (L) \neq 0 \).
	\end{proposition}

	\begin{proof}
		Let \( C \) be a \(2\)-colouring associated to the twisted orientation as per \Cref{Lem:twistedodd}. Denote by \( n^{\prime}_- \) the number of negative crossings with respect to the twisted orientation of \(D\), and \( n_- \) the number with respect to the orientation (an example is given in \Cref{Fig:3}). The homological grading of \( HBN(L) \) is defined by shifting the height of smoothings by \( n^{\prime}_- \), while that of \( \dbn (L)  \) shifts the height by \( n_- \). Notice that \( n_- = n^e_- + n^o_- \). \Cref{Lem:twistedodd} implies that
		\begin{equation*}
			\begin{aligned}
				n^o_+ &= n^t_- \\
				n^o_- &= n^t_+.
			\end{aligned}
		\end{equation*}
		Therefore \( n^{\prime}_- = n^e_- + n^t_- \), and \( n_- - n^{\prime}_- = n^o_- - n^t_- = n^t_+ - n^t_- \). Combining this with \cite[Proposition 2.17]{Chen2025} completes the proof.
	\end{proof}

	\subsubsection*{The spectral sequences are distinct}
	
	Despite the relationships of \Cref{Prop:chdbn_rank,Prop:chdbn_support} the spectral sequence defined by Chen and that of \Cref{Thm:dbnss} do not contain equivalent information. They are typically distinguished by their \( E_2 \) pages, and by the number of nontrivial pages.
	
	\begin{remark}\label{re:quantumshift}
		Recall that the two theories have different grading conventions. The necessary shift in homological degree (when passing from Chen's theory to the doubled theory) is given in \Cref{Prop:chdbn_support} as \( n^t_+ - n^t_- \). By a similar argument one can determine the necessary shift in quantum degree to be \( 3 ( n^t_+ - n^t_- ) \).
	\end{remark}
	
	For the knot given in \Cref{Fig:3} the sequences both have two nontrivial pages, but are distinguished by the (suitably shifted) quantum degree support of their \( E_2 \) pages.
	
	Further, for the knot given below, labelled \(4_1\) in Drobotukhina's table \cite{Drobotukhina94}, the sequences have a distinct number of nontrivial pages. Here we present their \(E_2\) pages, with that of the doubled Bar-Natan sequence of the left and Chen's sequence on the right:
	\begin{equation*}
	\begin{matrix}
	\raisebox{25pt}{\includegraphics{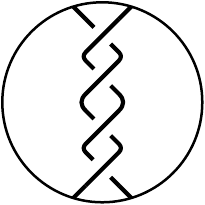}} & \includegraphics[scale=0.75]{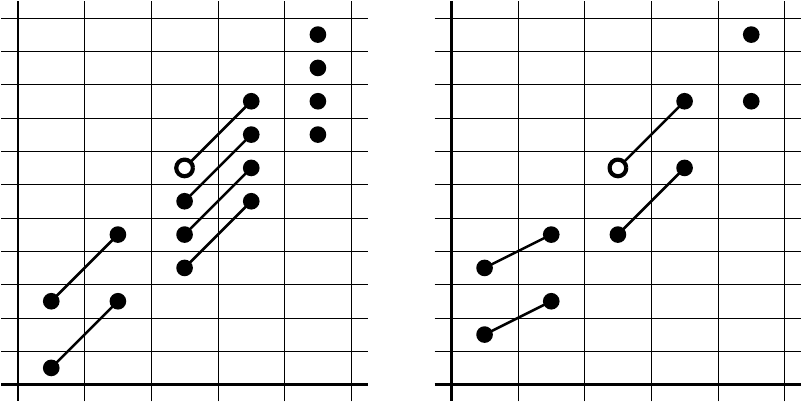}
	\end{matrix}
	\end{equation*}
	As before, \( \bullet, \circ \) denotes a copy of \( \mathbb{Z}_2 \). Differential terms are denoted by an edge between two generators, and axes are labelled by the integers. The hollow generator \( \circ \) in the right-hand grid is at bidegree \( (-2,-4) \). We compute \( n^t_- = 0 \), \( n^t_+ = 4 \), so that the hollow generator in the left-hand grid is at bidegree \( (2,8) \).
	
	Observe that the doubled Bar-Natan spectral sequence has two nontrivial pages, while Chen's spectral sequence has three.
	
	The index of the page at which the classical Bar-Natan sequence collapses places bounds on various geometric quantities; see, for example, \cite{Alishahi2019,Caprau2021}. It is therefore interesting to fully determine any relationship between the sequences considered here.
	
	\begin{question}\label{Q:pages}
		Let \(K\) be a nullhomologous knot such that Chen's spectral sequence has \(A\) nontrivial pages and the doubled Bar-Natan spectral sequence has \(B\) nontrivial pages.
		
		Is \( A \geq B \)? If so, can \( A - B \) be made arbitrarily large? If not, can \( B - A \) be made arbitrarily large?
	\end{question}
	
	\section{Functoriality}\label{Sec:properties}
	
	The doubled invariants defined in \Cref{Sec:dkh,Sec:dlee,Sec:dbn} enjoy functoriality as in the classical case. In \Cref{Sec:cobordisms} we show that concordances induce isomorphisms on doubled Lee and Bar-Natan homology. Together with appropriate results concerning the quantum grading support, this is applied in \Cref{Sec:ras_inv} to extend the Rasmussen invariant to \( \rp^3 \). We compare these new invariants to those previously defined by Chen and Manolescu-Willis. We conclude by describing the class of cobordisms on which the new Rasmussen invariants place genus bounds in \Cref{Sec:genusbounds}.
	
	Many of the constructions and proofs of this section have counterparts in the related setting of virtual links, as described in \cite[Section 3.2]{Rushworth2017}.
	
	\subsection{Cobordism maps}\label{Sec:cobordisms}
	
	Let \( L_0 \), \(L_1\) be links. For our purposes a \emph{cobordism from \(L_0\) to \(L_1\)} is a smoothly properly embedded compact surface \( S \hookrightarrow \rp^3 \times [0,1] \) such that \( \partial S = L_0 \sqcup L_1 \). We do not consider such cobordisms up to any notion of isotopy, so that \(S\) prescribes diagrams \( D_0 \), \( D_1 \) of \( L_0 \) and \(L_1\), respectively. As such we may alternatively refer to \(S\) as a cobordism from \( D_0 \) to \(D_1\).
	
	As is customary we first define the maps associated to simple cobordisms, from which those associated to general cobordisms are constructed.
	
	\begin{definition}[c.f.\ \cite{Carter1993}]\label{Def:elementary}
		Let \( S \) be a cobordism from \( D_0 \) to \( D_1 \). We say that \( C \) is \emph{elementary} if \( D_1 \) is obtained from \( D_0 \) via:
		\begin{enumerate}[(i)]
			\item A Reidemeister move (as given in \Cref{Prop:rm}).
			\item The birth of an unknotted, unlinked, homologically trivial component.
			\item The death of an unknotted, unlinked, homologically trivial component.
			\item An oriented saddle.
		\end{enumerate}
	\end{definition}

	\begin{definition}\label{Def:elementary_maps}
	Let \( S \) be an elementary cobordism from \(D_0 \) to \(D_1\), with \( D_0 \neq \emptyset \). It follows that \( D_0 \) and \(D_1\) are identical outside of a disc neighbourhood \( \nu \). There is a map of cubes \( \llbracket D_0 \rrbracket \rightarrow \llbracket D_1 \rrbracket \) defined by replacing \( \nu \cap D_0 \) with \( \nu \cap D_1 \) in every smoothing of \( D_0 \). This in turn induces a map
	\begin{equation*}
		\cdkh ( S ) : \cdkh ( D_0 ) \rightarrow \cdkh ( D_1 )
	\end{equation*}
 	as follows.
 	
 	The maps associated to Reidemeister moves are defined in essentially identical fashion to the classical case; see, for example, \cite[Section 8]{Bar-Natan2005} or \cite[Section 2.3]{Hayden2024}.
 	
 	If \( S \) is a saddle the map \( \cdkh ( S ) \) is defined on states as \( m \), \( \Delta \), or \( \eta \), as dictated by the map of cubes.
 	
 	If \( S \) is a birth the map \( \cdkh ( S ) \) is given by \( \iota \otimes 1 \otimes \cdots \otimes 1 \), where
 		\begin{equation*}
 			\begin{aligned}
 				\iota : R &\rightarrow A \\
 				1 &\mapsto v^{u/\ell}_{+}
 			\end{aligned}
 		\end{equation*}
 		with superscript determined by that of the argument state e.g.\ \( \iota ( 1 ) \otimes \vup = v^{u}_{+-} \).
 	
 	If \( S \) is a death the map \( \cdkh ( S ) \) is given by \( \epsilon \otimes 1 \otimes \cdots \otimes 1 \), where
 		\begin{equation*}
 			\begin{aligned}
 				\epsilon : A &\rightarrow R \\
 				\vup, \vlp &\mapsto 0 \\
 				\vum, \vlm &\mapsto 1.
 			\end{aligned}
 		\end{equation*}
 	\end{definition}
 
 	\begin{remark}\label{Remark:1}
 		The condition that \( D_0 \neq \emptyset \) in \Cref{Def:elementary_maps} is required only to ensure that \( \iota \) is well-defined. This condition can be removed by replacing the ground ring \(R\) with
 		\begin{equation*}
 			\mathcal{S} \coloneq R \oplus R \lbrace -1 \rbrace = R [ z ] / \langle z^2 - 1 \rangle
 		\end{equation*}
 		graded as \( p (1) = 0 \), \( p(z) = -1 \). Setting \( v^{\ell}_{\pm} = z v_{\pm} \) we have \( \mathcal{A} = \langle v_+, v_- \rangle_{\mathcal{S}} \) and may define 
 		\begin{equation*}
 			\iota (1) = v_+, \quad \iota (z) = z v_+
 		\end{equation*}
 		and
 		\begin{equation*}
 			\epsilon ( v_+ ) = \epsilon ( z v_+ ) = 0, \quad \epsilon ( v_- ) = 1, \quad \epsilon ( z v_- ) = z.
 		\end{equation*}
 		With this definition \( \iota \) and \( \epsilon \) are \( S \)-linear, but \( \eta \) is not e.g.\ \( \eta ( v_- ) = z v_- \neq 0 = \eta ( z v_- ) \). It follows that the tuple \( \left( A, m, \Delta, \eta, \epsilon, \iota \right) \) is not an unoriented TQFT as described by Turaev-Turner \cite{Turaev2006} (the construction is also not multiplicative with respect to disjoint union). There is a wealth of generalised TQFT constructions e.g.\ \cite{Blanchet2016,Turzillo2020}: do any of them provide a natural categorical framework for the doubled theories?
 		
 		Curiously, the map \( \eta_2 \) of \Cref{Def:cdbn} is indeed \( S \)-linear as it is multiplication by \(z\) (however, the Bar-Natan theory still fails to be multiplicative).
 	\end{remark}
 
 	That \( \cdkh (S)  \) is a chain map is verified exactly as in the classical case. We denote the induced map on doubled Khovanov homology as
 	\begin{equation*}
 		\dkh ( S ) : \dkh ( L_0 ) \rightarrow \dkh ( L_1 ).
 	\end{equation*}
 
 	A general cobordism \( S \) can be decomposed (nonuniquely) into a sequence of elementary cobordisms. Let
 	\begin{equation*}
 		S = S_0 \cup \cdots \cup S_n
 	\end{equation*}
 	be such a decomposition and define
 	\begin{equation*}
 		\dkh ( S ) \coloneqq \dkh ( S_n ) \circ \cdots \circ \dkh ( S_0 ).
 	\end{equation*}
 	We expect that \( \dkh ( S ) \) does not depend on the decomposition (i.e.\ that it is invariant under isotopy of \(S\) fixing the boundary pointwise), but as any dependence does not affect the results of this section we shall not investigate further.
 	
 	We define cobordism maps on doubled Lee and Bar-Natan homologies in identical fashion as for doubled Khovanov homology:
 	\begin{itemize}
 		\item Replacing \( m \), \( \Delta \), \( \eta \) in \Cref{Def:elementary_maps} with \( m ' \), \( \Delta ' \), \( \eta ' \) of \Cref{Def:cdlee} yields the map \( \dkh ' ( S ) : \dkh' ( L_0 ) \rightarrow \dkh' ( L_1 ) \).
 		\item Replacing \( m \), \( \Delta \), \( \eta \) in \Cref{Def:elementary_maps} with \( m_2 \), \( \Delta_2 \), \( \eta_2 \) of \Cref{Def:cdbn} yields the map \( \dbn ( S ) : \dbn ( L_0 ) \rightarrow \dbn ( L_1 ) \).
 	\end{itemize}
 
 	Notice that \( \dkh ( S ) \) is bigraded of degree \( \left( 0 , \chi ( C ) \right) \), while \( \dkh ' ( S ) \) and \( \dbn ( S ) \) are \(i\)-graded of degree \( 0 \) and \( j \)-filtered of degree \( \chi ( S ) \).

	Unlike in the classical case the maps \( \dkh ' ( S ) \) and \( \dbn ( S ) \) may be zero, even if \( S \) is orientable. Happily, concordances between \(2\)-colourable links induce nonzero maps, where for our purposes a cobordism is a concordance if it is a disjoint union of properly embedded annuli (one for each link component).
	
	\begin{lemma}[c.f.\ Theorem 3.21 \cite{Rushworth2017}]\label{Lem:nonzero}
		Let \( S \) be a concordance from \( L_0 \) to \( L_1 \) where \( L_0 \) is \(2\)-colourable. The maps \( \dkh ' ( S ) \) and \( \dbn ( S ) \) are nonzero. Moreoever, they map canonical generators to canonical generators (up to sign in the case of \(\dkh ' ( S )\)).
	\end{lemma}

	\begin{proof}
		Pick a decomposition of \( S \) into elementary cobordisms
		\begin{equation*}
			S = S_0 \cup \cdots \cup S_n
		\end{equation*}
		where \( \partial S_i = D_i \sqcup D_{i+1} \), and let \( f \in \lbrace \dkh ' ( S ), \dbn ( S ) \rbrace \).
		
		If there exists \( D_j \) that is not \(2\)-colourable then \( f \) is the zero map by \Cref{Cor:rank,Cor:rankbn}. We begin by proving that such a diagram is not encountered.
		
		If \( S_i \) is a birth or death of a component then \( D_{i+1} \) is \(2\)-colourable if and only if \( D_i \) is. If \( S_i \) is a saddle between two distinct components and \( D_i \) has no degenerate components then \( D_{i+1} \) has no such components either, so that \( D_{i+1} \) is \(2\)-colourable.
		
		Suppose that \( S_i \) is a saddle that splits one component into two. The resulting components of \( D_{i+1} \) are either both degenerate or both nondegenerate. If the former, observe that at least one of these components must be involved in another saddle (that merges two components) at some later point, as \( S \) is a concordance. However, the compactness of \(S\), and the fact that degenerate components cannot be removed by deaths, implies that this cannot occur.
		
		Recall that Reidemeister moves preserve \(2\)-colourability. Together with the above observations this implies that \( D_i \) is \(2\)-colourable for all \( 1 \leq i \leq n+1 \).
		
		Let \( g \in \lbrace \dkh ' ( S_i ), \dbn ( S_i ) \rbrace  \) for some \( 1 \leq i \leq n \). We claim that \( g \) maps canonical generators to canonical generators (up to sign in the case of \( \dkh ' ( S_i ) \)).
		
		If \( S_i \) is a birth or death of a component then the claim is verified by observing that
		\begin{equation*}
			\begin{aligned}
				\iota ( 1 ) &= o^{u/\ell} + p^{u/\ell} = a^{u/\ell} + b^{u/\ell} \\
				\epsilon ( o^{u/\ell} ) &= - \dfrac{1}{2} \\
				\epsilon ( p^{u/\ell} ) &= \dfrac{1}{2} \\
				\epsilon ( a^{u/\ell} ) &= \epsilon ( b^{u/\ell} ) = 1.
			\end{aligned}
		\end{equation*}
		
		Let \( S_i \) be a saddle between two distinct components. Consulting \Cref{Def:cdlee,Def:cdbn} we see that a canonical generator is sent to a canonical generator under \( g \) if the circles involved in the saddle are the same colour, and is otherwise sent to zero.
		
		Let \( S_i \) be a saddle that splits one component into two. The circles of a canonical generator involved in the saddle are coloured distinct colours if and only if the resulting components of \( D_{i+1} \) are degenerate. As established above such pairs of degenerate components do not occur in \( S \), so \( g \) sends canonical generators to (nonzero multiples of) canonical generators.
		
		If \( S_i \) is a Reidemeister move the arguments given by Rasmussen in the case of classical Lee homology can be applied here \emph{mutatis mutandis} \cite[Section 6]{Rasmussen2010}.
		
		It follows that \( f \) is nonzero, as it is the composition of the maps \( g \in \lbrace \dkh ' ( S_i ), \dbn ( S_i ) \rbrace \) for \( 1 \leq i \leq n \).
	\end{proof}

	Pardon's result on Lee homology \cite[Theorem 1.2, (5)]{Pardon2012} therefore extends to both doubled Lee and Bar-Natan homologies.

	\begin{corollary}\label{Cor:cinv}
		Doubled Lee and Bar-Natan homology are concordance invariants.
	\end{corollary}
	
	\subsection{Rasmussen invariants}\label{Sec:ras_inv}
	We have now established that the rank and homological grading support of both doubled Lee and Bar-Natan homologies are tightly controlled, and that these homologies are concordance invariants. The final ingredient required to extract Rasmussen invariants from these homologies concerns their quantum grading support.
	
	Let \( D \) be a diagram of a link \( L \). As mentioned in \Cref{Sec:dlee} both \( \cdkh ' (D) \) and \( \cdbn ( D ) \) are \(j\)-filtered where
	\begin{equation*}
		F^k \cdkh ' (D) \coloneqq \lbrace x \in \cdkh ' (D) ~|~ j (x) \geq k \rbrace
	\end{equation*}
	and \( F^k \cdbn (D) \) is defined similarly. Denote by \( s \) the grading on \( \dkh ' (L) \) and \( \dbn ( L ) \) induced by the filtration.
	
	\begin{proposition}\label{Prop:qsupport}
		Let \( K \) be a knot. The \(s\)-grading supports of \( \dkh ' (K) \) and \( \dbn ( K ) \) are respectively of the form
		\begin{equation*}
			\lbrace u + 1, u, u - 1 , u - 2 \rbrace
		\end{equation*}
		for some (possibly distinct) \( u \in \mathbb{Z} \).
	\end{proposition}
	Note that the supports of \( \dkh ' (K) \) and \( \dbn ( K ) \) are not known to be equal in general.
	
	\begin{proof}
		As with the proofs of invariance of doubled Lee and Bar-Natan homologies, we can follow the classical verifications, given in \cite[Section 3.1]{Rasmussen2010}, in essentially identical fashion. See also the analogous results in the related setting of virtual knots \cite[Lemma 4.2, Proposition 4.4]{Rushworth2017}.
	\end{proof}

	The quantum grading supports of \( \dkh ' (K) \) and \( \dbn ( K ) \) are therefore each equivalent to an integer (not necessarily even, unlike in the classical case).

	\begin{definition}\label{Def:ras}
		Let \( K \) be a knot and \( R \) a commutative unital ring in which \(2\) is invertible, or \( \mathbb{Z}_2 \). The \emph{doubled Rasmussen invariant (over \(R\)) of \(K\)} is defined
		\begin{equation*}
			ds_R ( K ) \coloneqq \text{max} \lbrace s ( x ) \in H ~|~ x \neq 0 \rbrace - 1
		\end{equation*}
		where \( H = \dbn (K) \) if \( R = \mathbb{Z}_2 \) and \( H = \dkh ' (K) \) otherwise.
	\end{definition}
	In other words, \( ds_R ( K ) = u \) of \Cref{Prop:qsupport}. Note that \Cref{Cor:cinv} implies that the doubled Rasmussen invariants are concordance invariants. In particular, the doubled Rasmussen invariants obstruct sliceness.
	
	\subsubsection*{Comparison with other extensions}
	The classical Rasmussen invariants over \( \mathbb{Q} \) and \( \mathbb{Z}_2 \) are distinct on the knot \( 14n19265 \) \cite[Remark 6.1]{Lipshitz2014}. Including this knot in a \(3\)-ball, which is in turn included in \( \rp^3 \), yields an example of knot on which the doubled Rasmussen invariants over \( \mathbb{Q} \) and \( \mathbb{Z}_2 \) differ, that we now demonstrate.
	
	We say that a knot \( K \) in \( \rp^3 \) is a \emph{local} if it can be contained in a \(3\)-ball. The doubled invariants of local knots are of constrained form.
	
	\begin{proposition}[c.f.\ Proposition 2.5 \cite{Rushworth2017}]\label{Prop:local}
		Let \( D \) be a knot diagram that does not intersect the boundary of the ambient disc\footnote{That is, as a tangle it has only closed components.}. Then
		\begin{equation*}
			\cdkh ( D ) = \ckh ( D ) \oplus \ckh ( D ) \lbrace -1 \rbrace
		\end{equation*}
		where \( \ckh ( D ) \) is the classical Khovanov complex of \( D \), when considered as a diagram in the plane.
	\end{proposition}

	\begin{proof}
		As \( D \) does not intersect the boundary of the ambient disc no \( \eta \) maps occur in \( \cdkh ( D ) \), so that it splits as the required direct sum.
	\end{proof}
	
	A local knot \( K \) in \( \rp^3 \) uniquely determines a knot in \( S^3 \), by including the \(3\)-ball that contains \(K\) in \( S^3 \).
	
	\begin{corollary}\label{Cor:localras}
		Let \( K \) be a local knot. Then the doubled Rasmussen invariant of \( K \) is equal to the classical Rasmussen invariant of the knot in \( S^3 \) determined by \(K\) (over any coefficient ring).
	\end{corollary}
	It follows that including \( 14n19265 \) in \( \rp^3 \) yields a knot on which the doubled Rasmussen invariants over \( \mathbb{Q} \) and \( \mathbb{Z}_2 \) differ.
	
	Denote by \( s_{\mathbb{Q}} ( K ) \) and \( s_{\mathbb{Z}_2} ( K ) \) the extensions of the Rasmussen invariant to \( \rp^3 \) defined by Manolescu-Willis and Chen, respectively \cite{Manolescu25,Chen2025}.
	
	For the knot \( 2_1 \) of \Cref{Fig:2} we have \( ds_{\mathbb{Q}} ( 2_1 ) = - 5 \) and \( s_{\mathbb{Q}} ( 2_1 ) = - 2 \), so that the doubled invariant and that of Manolescu-Willis are distinct.
	
	In \Cref{Sec:dbn} it is demonstrated that Chen's spectral sequence and the doubled Bar-Natan spectral sequence are distinct. However, there are no examples known to the author of knots \( J \) such that
	\begin{equation*}
		ds_{\mathbb{Z}_2} ( J ) \neq s_{\mathbb{Z}_2} ( J ) + 3 \left( n^t_+ - n^t_- \right)
	\end{equation*}
	where \(3 \left( n^t_+ - n^t_- \right)\) accounts for the differing grading conventions (see \Cref{re:quantumshift}). In other words, there are no known examples in which the \( E_{\infty} \) pages carry inequivalent information (recall that their rank and homological degree supports are necessarily equivalent, as per \Cref{Cor:rankbn,Cor:homsupportbn}).
	
	However, owing to the distinct degrees of differential terms in the two sequences -- for example, Chen's sequence may contain terms of odd degree while the doubled sequence cannot -- it is reasonable to suspect that the \( E_{\infty} \) pages carry inequivalent information in general.
	
	\begin{question}\label{Q:z2ras}
		Does there exist a knot \(J\) such that
		\begin{equation*}
			ds_{\mathbb{Z}_2} ( J ) \neq s_{\mathbb{Z}_2} ( J ) + 3 \left( n^t_+ - n^t_- \right)?
		\end{equation*}
		
		In particular, does there exist a \(J\) such that one of \( ds_{\mathbb{Z}_2} ( J ) \), \( s_{\mathbb{Z}_2} ( J ) \) obstructs sliceness but the other does not?
	\end{question}

	\subsection{Genus bounds}\label{Sec:genusbounds}
	We conclude by describing the class of cobordisms on which the doubled Rasmussen invariant imposes genus bounds. The classical Rasmussen invariant, and that defined by Manolescu-Willis, impose such bounds on all cobordisms; in the case of Chen's extension the relevant cobordisms are known as twisted orientable \cite[Definition 3.1]{Chen2025}.
	
	Let \(S\) be an elementary cobordism from \( D_0 \) to \( D_1 \). Recall that \( D_0 \) and \(D_1\) agree outside of a disc neighbourhood \( \nu \). Given \(2\)-colourings \( C_0 \), \(C_1\) of \( D_0 \), \(D_1\), respectively, we say that \( C_0  \) may be \emph{propagated through \(S\) (to yield \( C_1 \))} if \( C_0 \) and \(C_1\) agree outside \( \nu \).
	
	It is easy to see that if \( S \) is an elementary cobordism defined by a Reidemeister move, or a birth or death (of an unknotted, unlinked component), then every \( 2 \)-colouring of the initial diagram propagates. If \( S \) is a saddle we claim that a \(2\)-colouring of the initial diagram propagates if and only if the arcs contained in \( \nu \) are of the same colour. It is straightforward to verify this in the case that \( S \) merges two components into one.
	
	If \( S \) splits one component into two observe that the arcs contained in \( \nu \) are of distinct colours if and only if the linking number of the newly created components is odd. It follows that if \( D_0 \) is \(2\)-colourable then \( D_1 \) is \(2\)-colourable if and only if the arcs contained in \( \nu \) are of the same colour (otherwise \( D_1 \) would contain a degenerate component). For further details in the analogous case of virtual links see \cite[Lemma 3.24]{Rushworth2017}.

	\begin{definition}\label{2colcob}
		Let \( S \) be a cobordism and
		\begin{equation*}
			S = S_0 \cup \cdots \cup S_n
		\end{equation*}
		a decomposition of it into elementary cobordisms \( S_i \), from \( D_i \) to \( D_{i+1} \). We say that \( S \) is \emph{\(2\)-colourable} if there exists a \(2\)-colouring, \(C_i\), of \( D_i \) for all \( 0 \leq i \leq n  \), such that \( C_i \) is propagated through \( S_i \) to yield \( C_{i+1} \).
	\end{definition}
	In other words, \( S \) is \(2\)-colourable if there exists a \(2\)-colouring of \( D_0 \) that can be propagated through \( S \) in its entirety.
	
	The proof of \Cref{Lem:nonzero} establishes that concordances between \(2\)-colourable links are \(2\)-colourable. The arguments used there can be adapted in a straightforward way to prove that genus \(0\) cobordisms (between \(2\)-colourable links) are \(2\)-colourable also. This ultimately follows from the fact that their Reeb graphs are trees.
	
	If a non-\(2\)-colourable link appears in a \( S \) then \( S \) is certainly not \(2\)-colourable itself. Further, there are cobordisms in which only \(2\)-colourable links appear, but the cobordism is not \(2\)-colourable itself e.g.\
	\begin{equation}\label{Fig:4}
		\begin{matrix}
			\includegraphics[scale=0.75]{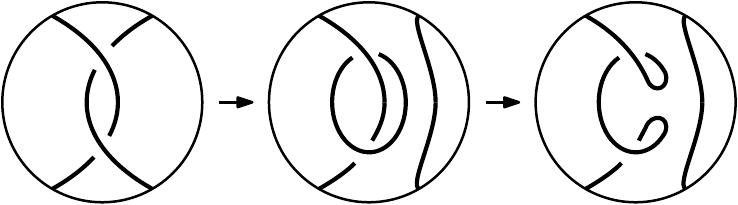}
		\end{matrix}
	\end{equation}
	The leftmost diagram has two \(2\)-colourings, that are easily seen not to propagate to the rightmost diagram.
	
	\begin{theorem}\label{Thm:genusbound}
		Let \( K \) be a knot and \( S \) a \(2\)-colourable cobordism from \(K\) to the (homologically trivial) unknot. Then
		\begin{equation*}
			\left| ds_{R} (K) \right| \leq 2 g ( S ).
		\end{equation*}
	\end{theorem}

	\begin{proof}
		The proof that \( \dkh ' ( S ) \) and \( \dbn ( S ) \) are nonzero and map canonical generators to (nonzero multiples of) canonical generators is essentially identical to that of \Cref{Lem:nonzero}. With this in hand the argument given by Rasmussen in the classical case \cite[Proof of Theorem 1]{Rasmussen2010} can be applied directly.
	\end{proof}

	\Cref{Fig:4} describes a genus-\(1\) cobordism beginning at the knot \( 2_1 \) given in \Cref{Fig:2}. As noted above, we have \( ds_{\mathbb{Q}} ( 2_1 ) = - 5 \) and \( s_{\mathbb{Q}} ( 2_1 ) = - 2 \). It follows from \cite[Theorem 1.2]{Manolescu25} that the cobordism of \Cref{Fig:4} is genus-minimizing. As \( ds_{\mathbb{Q}} ( 2_1 ) = - 5 \) \Cref{Thm:genusbound} implies that \(2\)-colourable cobordisms are not genus-minimizing. That is, \(2\)-colourable cobordism form a proper subset of the set of all cobordisms.
	
	In \cite[Example 4.11]{Chen2025} it is demonstrated that twisted orientable cobordisms also form such a proper subset. Moreover, it is shown that for all \(N \in \mathbb{N} \) there exists a knot \( K \) such that
	\begin{equation}\label{Eq:arbitrary}
		\text{min} \left( \left\lbrace g ( S ) ~\left| ~ \begin{matrix}
			\text{\(S\) a twisted orientable} \\ \text{cobordism from \(K \) to \(U\)}
		\end{matrix}\right. \right\rbrace \right) \geq \text{min} \left( \left\lbrace g ( S ) ~\left| ~ \begin{matrix}
		\text{\(S\) a cobordism} \\ \text{from \(K \) to \(U\)}
	\end{matrix}\right. \right\rbrace \right) + N
	\end{equation}
	for \( U \) the homologically trivial unknot. That is, the minimum genus of twisted orientable cobordisms is arbitrarily larger than the standard slice genus.
	
	The same is true if twisted orientable is replaced with \(2\)-colourable in \Cref{Eq:arbitrary}. In particular, applying \Cref{Thm:genusbound} to the knots given in \cite[Figure 22]{Chen2025} establishes that the minimal genus of \( 2 \)-colourable cobordisms is arbitrarily larger than the standard slice genus.
	
	However, a general relationship between twisted orientable and \(2\)-colourable cobordisms remains unknown.
	\begin{question}
		\begin{enumerate}
			\item Is the minimal genus of \( 2 \)-colourable cobordisms less than that of twisted orientable cobordisms in general, or vice versa?
			\item For all \( M \in \mathbb{N} \) does there exist a knot \( K \) such that
			\begin{equation*}
				\left| \text{min} \left( \left\lbrace g ( S ) ~\left| ~ \begin{matrix}
					\text{\(S\) a twisted orientable} \\ \text{cobordism from \(K \) to \(U\)}
				\end{matrix}\right. \right\rbrace \right) - \text{min} \left( \left\lbrace g ( S ) ~\left| ~ \begin{matrix}
					\text{\(S\) a \(2\)-colourable} \\ \text{cobordism from \(K \) to \(U\)}
				\end{matrix}\right. \right\rbrace \right) \right| \geq M?
			\end{equation*}
		\end{enumerate}
	\end{question}

	\bibliographystyle{alpha}
	\bibliography{library}

\newcommand{\etalchar}[1]{$^{#1}$}
\begin{thebibliography}{BCGPM16}

\bibitem[Ali19]{Alishahi2019}
A.~Alishahi.
\newblock Unknotting number and {K}hovanov homology.
\newblock {\em Pacific J. Math.}, 301(1):15--29, 2019.

\bibitem[APS04]{APS}
M.~M. Asaeda, J.~H. Przytycki, and A.~S. Sikora.
\newblock Categorification of the {K}auffman bracket skein module of
  {$I$}-bundles over surfaces.
\newblock {\em Algebr. Geom. Topol.}, 4:1177--1210, 2004.

\bibitem[BCGPM16]{Blanchet2016}
C.~Blanchet, F.~Costantino, N.~Geer, and B.~Patureau-Mirand.
\newblock Non-semi-simple {TQFT}s, {R}eidemeister torsion and {K}ashaev's
  invariants.
\newblock {\em Adv. Math.}, 301:1--78, 2016.

\bibitem[BN02]{Bar-Natan2002}
D.~Bar-Natan.
\newblock {On Khovanov's categorification of the Jones polynomial}.
\newblock {\em Algebraic \& Geometric Topology}, 2(1):337--370, 2002.

\bibitem[BN05]{Bar-Natan2005}
D.~Bar-Natan.
\newblock Khovanov's homology for tangles and cobordisms.
\newblock {\em Geom. Topol.}, 9:1443--1499, 2005.

\bibitem[CGL{\etalchar{+}}21]{Caprau2021}
C.~Caprau, N.~Gonz\'{a}lez, C.~R.~S. Lee, A.~M. Lowrance, R.~Sazdanovi\'{c},
  and .~Zhang.
\newblock On {K}hovanov homology and related invariants.
\newblock In {\em Research directions in symplectic and contact geometry and
  topology}, volume~27 of {\em Assoc. Women Math. Ser.}, pages 273--292.
  Springer, Cham, [2021] \copyright 2021.

\bibitem[Che21]{Chen2021}
D.~Chen.
\newblock Khovanov-type homologies of null homologous links in $\mathbb{RP}^3$,
  2021.
\newblock \url{arXiv:2104.04779}.

\bibitem[Che25]{Chen2025}
D.~Chen.
\newblock Bar-{N}atan homology for null homologous links in {$\Bbb{RP}^3$}.
\newblock {\em Pacific J. Math.}, 334(2):233--268, 2025.

\bibitem[CS93]{Carter1993}
J.~S. Carter and M.~Saito.
\newblock Reidemeister moves for surface isotopies and their interpretation as
  moves to movies.
\newblock {\em J. Knot Theory Ramifications}, 2(3):251--284, 1993.

\bibitem[Dro90]{Drob90}
Y.~V. Drobotukhina.
\newblock An analogue of the {J}ones polynomial for links in {${\bf R}{\rm
  P}^3$} and a generalization of the {K}auffman-{M}urasugi theorem.
\newblock {\em Algebra i Analiz}, 2(3):171--191, 1990.

\bibitem[Dro94]{Drobotukhina94}
J.~Drobotukhina.
\newblock Classification of links in {${\bf R}{\rm P}^3$} with at most six
  crossings.
\newblock In {\em Topology of manifolds and varieties}, volume~18 of {\em Adv.
  Soviet Math.}, pages 87--121. Amer. Math. Soc., Providence, RI, 1994.

\bibitem[Gab13]{Gabrovsek2013}
B.~Gabrov\v{s}ek.
\newblock The categorification of the {K}auffman bracket {S}kein module of
  {$\Bbb R\rm P^3$}.
\newblock {\em Bull. Aust. Math. Soc.}, 88(3):407--422, 2013.

\bibitem[HS24]{Hayden2024}
K.~Hayden and I.~Sundberg.
\newblock Khovanov homology and exotic surfaces in the 4-ball.
\newblock {\em J. Reine Angew. Math.}, 809:217--246, 2024.

\bibitem[IMN13]{Ilyutko2013}
D.~P. Ilyutko, V.~O. Manturov, and I.~M. Nikonov.
\newblock Parity in knot theory and graph-links.
\newblock {\em J. Math. Sci. (N. Y.)}, 193(6):809--965, 2013.

\bibitem[Kau04]{Kauffman2004b}
Louis~H Kauffman.
\newblock {A self-linking invariant of virtual knots}.
\newblock {\em Fundamenta Mathematicae}, 184(1):135--158, 2004.

\bibitem[KMN24]{Kauffman2024}
L.~H. Kauffman, R.~Mishra, and V.~Narayanan.
\newblock Knots in $\mathbb{RP}^3$, 2024.
\newblock \url{arXiv:2401.06050}.

\bibitem[Lee05]{Lee2005}
E.~S. Lee.
\newblock {An endomorphism of the Khovanov invariant}.
\newblock {\em Advances in Mathematics}, 197(2):554--586, 2005.

\bibitem[LS14]{Lipshitz2014}
R.~Lipshitz and S.~Sarkar.
\newblock A refinement of {R}asmussen's {$S$}-invariant.
\newblock {\em Duke Math. J.}, 163(5):923--952, 2014.

\bibitem[Man10]{Manturov2010}
V.~O. Manturov.
\newblock Parity in knot theory.
\newblock {\em Sb. Math.}, 201(5):693, 2010.

\bibitem[MW25]{Manolescu25}
C.~Manolescu and M.~Willis.
\newblock A {R}asmussen invariant for links in {$\Bbb{RP}^3$}.
\newblock {\em Trans. Amer. Math. Soc. Ser. B}, 12:789--830, 2025.

\bibitem[Nao06]{Naot2006}
G.~Naot.
\newblock The universal {K}hovanov link homology theory.
\newblock {\em Algebr. Geom. Topol.}, 6:1863--1892, 2006.

\bibitem[Par12]{Pardon2012}
J.~Pardon.
\newblock The link concordance invariant from {L}ee homology.
\newblock {\em Algebr. Geom. Topol.}, 12(2):1081--1098, 2012.

\bibitem[Ras10]{Rasmussen2010}
J.~Rasmussen.
\newblock {Khovanov homology and the slice genus}.
\newblock {\em Inventiones Mathematicae}, 182(2):419--447, 2010.

\bibitem[Ren24]{Ren2024}
Q.~Ren.
\newblock Slice genus bound in {${DTS}^2$} from {$s$}-invariant.
\newblock {\em Algebr. Geom. Topol.}, 24(7):4115--4125, 2024.

\bibitem[Rus18]{Rushworth2017}
W.~Rushworth.
\newblock {Doubled Khovanov Homology}.
\newblock {\em Canadian Journal of Mathematics}, 70:1130--1172, 2018.

\bibitem[Rus21]{Rushworth2021}
W.~Rushworth.
\newblock A parity for 2-colourable links.
\newblock {\em Osaka J. Math.}, 58(4):767--801, 2021.

\bibitem[RY25]{Ren2025}
Q.~Ren and H.~Yang.
\newblock Intrinsic {K}hovanov homology in $\mathbb{RP}^3$, 2025.
\newblock \url{arXiv:2510.00109}.

\bibitem[TT06]{Turaev2006}
Vladimir Turaev and Paul Turner.
\newblock {Unoriented topological quantum field theory and link homology}.
\newblock {\em Algebraic \& Geometric Topology}, 6(3):1069--1093, 2006.

\bibitem[Tur06]{Turner2006}
P.~R. Turner.
\newblock Calculating {B}ar-{N}atan's characteristic two {K}hovanov homology.
\newblock {\em J. Knot Theory Ramifications}, 15(10):1335--1356, 2006.

\bibitem[Tur20]{Turzillo2020}
A.~Turzillo.
\newblock Diagrammatic state sums for 2{D} pin-minus {TQFT}s.
\newblock {\em J. High Energy Phys.}, (3):019, 26, 2020.

\bibitem[Yan25]{Yang2025}
H.~Yang.
\newblock Instantons and {K}hovanov homology in {$\Bbb{RP}^3$}.
\newblock {\em Trans. Amer. Math. Soc.}, 378(7):4991--5009, 2025.

\end{thebibliography}
	
\end{document}